\documentclass[a4paper,11pt]{amsart}
\usepackage[english]{babel}
\usepackage[latin1]{inputenc}
\usepackage{amssymb}
\usepackage{amsmath,amscd}
\usepackage{mathtools}
\usepackage[paper=a4paper,left=3cm,right=3cm,top=3cm,bottom=2.3cm]{geometry}
\usepackage{enumerate}
\usepackage{url}
\usepackage{hyperref}

\newtheorem{thm}{Theorem}[section]
\newtheorem*{thm-non}{Theorem}
\newtheorem{lem}[thm]{Lemma}

\newtheorem{cor}[thm]{Corollary}
\theoremstyle{definition}
\newtheorem{defi}[thm]{Definition}
\newtheorem{rem}[thm]{Remark}

\begin{document}

\title{Hermitian Azumaya modules and arithmetic Chern classes}

\author{Fabian Reede}
\address{Department of Mathematics and Computer Studies\\Mary Immaculate College\\South Circular
  Road\\Limerick\\Ireland}
\address{Leibniz Universität Hannover, Institut für algebraische Geometrie, Welfengarten 1, 30167 Hannover}
\email{reede@math.uni-hannover.de}

\thanks{Part of this work was done during the author's stay at Mary Immaculate College in Limerick. This stay was supported by a research fellowship of the Deutsche Forschungsgemeinschaft (DFG)}
\maketitle

\begin{abstract}
We compute arithmetic Chern classes of sheaves on an arithmetic surface $X$ associated to a Hermitian Azumaya algebra.
\end{abstract}

\section*{Introduction}
Let $\pi: X\rightarrow Y$ be an arithmetic surface and let $\mathcal{A}$ be an Azumaya algebra on $X$. In \cite{reede} we introduced the notion of Hermitian Azumaya algebras and Hermitian Azumaya modules by equipping certain sheaves with Hermitian metrics. Using this, we defined a Deligne pairing $\overline{\left\langle \mathcal{M},\mathcal{N} \right\rangle_\mathcal{A} }$ for Hermitian $\mathcal{A}$-line bundles $\overline{\mathcal{M}}$ and $\overline{\mathcal{N}}$, which generalized the classical Deligne pairing for line bundles on $X$.

In this note we introduce arithmetic Chern classes for Hermitian Azumaya modules and compute the first arithmetic Chern class of the Deligne pairing for $\mathcal{A}$-line bundles.

Our main result shows, that we can compute the first arithmetic Chern class of the $\mathcal{A}$-Deligne pairing in terms of the arithmetic $\mathcal{A}$-Chern classes. Explicitly is says:
\begin{equation*}
\widehat{c}_1(\overline{\left\langle \mathcal{M},\mathcal{N} \right\rangle_\mathcal{A} })=-\pi_{*}(\widehat{c}_1^{\mathcal{A}}(\overline{\mathcal{M}})\,\widehat{c}_1^{\mathcal{A}}(\overline{\mathcal{N}})).
\end{equation*}

The structure of this paper is as follows: In section \ref{sec1} we recall some facts about Hermitian vector space, the construction of Hermitian inner products on associated vector spaces and isometric vector spaces. In section \ref{sec2} we recall the definition of Hermitian Azumaya algebras and compute the first arithmetic Chern class of a Hermitian Azumaya algebra. In the final section \ref{sec3} we introduce arithmetic Chern classes for Hermitian Azumaya modules and proof the main result.

\section{Hermitian vector spaces}\label{sec1}
\begin{defi}
A Hermitian vector space $\overline{V}$ is a pair $(V,h)$, where $V$ is a finite dimensional $\mathbb{C}$-vector space and $h: V\times V\rightarrow \mathbb{C}$ is a Hermitian inner product.
\end{defi}
The inner product $h$ induces the so called Riesz isomorphism between $V$ and the dual space $V^{\vee}=Hom_{\mathbb{C}}(V,\mathbb{C})$ defined by
\begin{equation*}
\theta_V: V\rightarrow V^{\vee}, v\mapsto h(-,v).
\end{equation*}
Note that this a conjugate linear isomorphism. Using this isomorphism we define the associated dual hermitian vector space $\overline{V^{\vee}}$ to be the pair $(V^{\vee},h^{\vee})$, where $h^{\vee}$ is the induced dual inner product defined by:
\begin{equation*}
h^{\vee}(f,f'):=\overline{h(\theta_V^{-1}(f),\theta_V^{-1}(f'))}\,\,\,\text{for}\,\,\, f,f'\in V^{\vee}
\end{equation*}
This inner product gives the Riesz isomorphism $\theta_{V^{\vee}}:V^{\vee}\rightarrow V^{\vee\vee}$.

Iterating this construction we get the induced bidual Hermitian vector space $\overline{V^{\vee\vee}}=(V^{\vee\vee},h^{\vee\vee})$.

Given two Hermitian vector space $\overline{V}=(V,h)$ and $\overline{W}=(W,k)$ we say $\Psi: V\rightarrow W$ induces an isometry of Hermitian vector spaces if $\Psi$ is an isomorphism of $\mathbb{C}$-vector spaces and further more $\Psi^{*}k=h$, that is we have, for any two $v,v'\in V$ the equality $h(v,v')=k(\Psi(v),\Psi(v'))$. 

\begin{lem}
The natural isomorphism $\iota: V\rightarrow V^{\vee\vee}$ induces an isometry of Hermitian vector spaces
\begin{equation*}
\iota: \overline{V} \xrightarrow{\sim} \overline{V^{\vee\vee}}.
\end{equation*}
\end{lem}
\begin{proof}
We note that we have $\iota = \theta_{V^{\vee}}\circ\theta_V$. This can be see as follows: for arbitrary $v\in V$ and $f\in V^{\vee}$ we have
\begin{align*}
((\theta_{V^{\vee}}\circ\theta_V)(v))(f)&=h^{\vee}(f,\theta_V(v))\\
&=\overline{h(\theta_V^{-1}(f),\theta_V^{-1}(\theta_V(v)))}\\
&=h(v,\theta_V^{-1}(f))\\
&=f(v)=\iota(v)(f).
\end{align*}
Using this fact we can compute that we have $\iota^{*}h^{\vee\vee}=h$, hence $\iota$ is an isometry between $\overline{V}$ and $\overline{V^{\vee\vee}}$.
\end{proof}
Equip the tensor product $V\otimes W$ with the induced tensor product metric $h\otimes k$ defined on elementary tensors by:
\begin{equation*}
h\otimes k(v\otimes w, v'\otimes w'):=h(v,v')k(w,w').
\end{equation*}
This construction defines the Hermitian vector space $\overline{V\otimes W}$. 
\begin{rem}
The Riesz isomorphism $\theta_{V\otimes W}$ induced by $h\otimes k$ has the following decomposition:
\begin{equation*}
\theta_{V\otimes W}=\alpha \circ (\theta_V\otimes\theta_W).
\end{equation*}
Here $\alpha: V^{\vee}\otimes W^{\vee}\rightarrow (V\otimes W)^{\vee}$ is the natural isomorphism defined on elementary tensors by 
\begin{equation*}
f\otimes g\mapsto(v\otimes w \mapsto f(v)g(w)).
\end{equation*}
\end{rem}

\begin{lem}\label{dual}
The natural isomorphism $\alpha: V^{\vee}\otimes W^{\vee}\rightarrow (V\otimes W)^{\vee}$ induces an isometry of Hermitian vector spaces
\begin{equation*}
\alpha: \overline{V^{\vee}\otimes W^{\vee}}\xrightarrow{\sim}\overline{(V\otimes W)^{\vee}}.
\end{equation*}
\end{lem}
\begin{proof}
We have to show that $\alpha^{*}(h\otimes k)^{\vee}=h^{\vee}\otimes k^{\vee}$. Using the aforementioned decomposition of $\theta_{V\otimes W}$ we have:
\begin{align*}
\alpha^{*}(h\otimes k)^{\vee}(f\otimes g,f'\otimes g')&=\overline{h\otimes k(\theta_{V\otimes W}^{-1}(\alpha(f\otimes g)),\theta_{V\otimes W}^{-1}(\alpha(f'\otimes g')))}\\
&=\overline{h\otimes k(\theta_V^{-1}\otimes\theta_W^{-1}(f\otimes g),\theta_V^{-1}\otimes\theta_W^{-1}(f'\otimes g'))}\\
&=\overline{h(\theta_V^{-1}(f),\theta_V^{-1}(f'))}\,\,\,\overline{k(\theta_W^{-1}(g),\theta_W^{-1}(g'))}\\
&=h^{\vee}(f,f')\,\,\,k^{\vee}(g,g')\\
&=h^{\vee}\otimes k^{\vee}(f\otimes g, f'\otimes g').
\end{align*}
\end{proof}
\begin{rem}\label{twist}
We also note that the natural isomorphism 
\begin{equation*}
V\otimes W \xrightarrow{\sim} W\otimes V
\end{equation*}
induces an isometry of Hermitian vector spaces.
\end{rem}

Composing all these natural isomorphisms we get a natural isomorphism
\begin{equation}\label{isos}
V^{\vee}\otimes W\cong V^{\vee}\otimes W^{\vee\vee}\cong W^{\vee\vee}\otimes V^{\vee} \cong (W^{\vee}\otimes V)^{\vee}
\end{equation}
which, as we showed, induces in fact an isometry of Hermitian vector spaces:
\begin{equation*}
\overline{V^{\vee}\otimes W}\xrightarrow{\sim} \overline{(W^{\vee}\otimes V)^{\vee}}.
\end{equation*}
Using the natural isomorphisms 
\begin{equation*}
V^{\vee}\otimes W\cong Hom(V,W)\,\,\text{and}\,\,W^{\vee}\otimes V\cong Hom(W,V)\,\,,
\end{equation*}
the isomorphism \ref{isos} gives rise to a natural isomorphism
\begin{equation*}
Hom(V,W) \xrightarrow{\sim} Hom(W,V)^{\vee}.
\end{equation*}
It is well known that this isomorphism is nothing but the trace map, i.e. it comes from the perfect pairing:
\begin{equation*}
Hom(V,W) \times Hom(W,V) \xrightarrow{(-)\circ(-)} Hom(V,V)\xrightarrow{\text{tr}} \mathbb{C}
\end{equation*} 
and it maps $\phi: V\rightarrow W\in Hom(V,W)$ to $\text{tr}((-)\circ \phi)\in Hom(W,V)^{\vee}$.

The natural isomorphism $V^{\vee}\otimes W\xrightarrow{\sim} Hom(V,W)$ also defines a Hermitian inner product on $Hom(V,W)$ which in turn defines the Hermitian vector space $\overline{Hom(V,W)}$.

Putting all these steps together we have proven:
\begin{thm}\label{trace}
The trace map $\text{tr}: Hom(V,W) \rightarrow Hom(W,V)^{\vee}$ induces an isometry of Hermitian vector spaces
\begin{equation*}
\text{tr}: \overline{Hom(V,W)}\xrightarrow{\sim} \overline{Hom(W,V)^{\vee}}.
\end{equation*}
\end{thm}

\begin{rem}
All results in this section can be immediately generalized to Hermitian vector bundles, that is pairs $(E,h)$, where $E$ is a complex vector bundle on a smooth manifold and $h$ is a Hermitian metric on $E$, that is a Hermitian inner product on each fiber.
\end{rem}

\section{Hermitian Azumaya algebras}\label{sec2}
\begin{defi}
We call a two-dimensional integral regular projective flat $\mathbb{Z}$-scheme $X$  an arithmetic surface and denote the structure morphism by $\pi: X\rightarrow Y$, here we define $Y:=Spec(\mathbb{Z})$.
\end{defi}

We denote the generic fiber of $\pi$ by $X_{\mathbb{Q}}$ and the base change to $\mathbb{C}$ by $X_{\mathbb{C}}$. Both $X_\mathbb{Q}$ and $X_\mathbb{C}$ are smooth curves over $\mathbb{Q}$ resp. $\mathbb{C}$.

Given a sheaf of $\mathcal{O}_X$-modules $\mathcal{F}$, then we denote the induced sheaves on $X_\mathbb{Q}$ and $X_\mathbb{C}$ by $\mathcal{F}_\mathbb{Q}$ and $\mathcal{F}_\mathbb{C}$. 

The associated Riemann surface $S$ of $X$ is given by $S=X_\mathbb{C}(\mathbb{C})$. $S$ comes endowed with a Hermitian metric $\omega$, which is Kähler. If $\mathcal{F}$ is a locally free $\mathcal{O}_X$-module, then we denote the induced vector bundle on $S$ by $F$.
\begin{defi}
A Hermitian vector bundle $\overline{\mathcal{E}}$ on $X$ is a pair $(\mathcal{E},h)$, where $\mathcal{E}$ is a locally free sheaf on $X$ and $h$ is a Hermitian metric on the induced vector bundle $E$ on $S$, which is invariant under the complex conjugation on $S$.
\end{defi}

\begin{rem}
If we use a Hermitian metric in the following, we will always assume that the metric is invariant with respect to the complex conjugation on $S$, see \cite[Definition IV.4.1.4.]{soul3}.
\end{rem}

Let $\mathcal{A}$ be an Azumaya algebra on the arithmetic surface $X$, that is a matrix algebra in the \'etale topology, then we have
\begin{equation}\label{morita}
\mathcal{A}_\mathbb{C}\cong \mathcal{E}nd_{\mathcal{O}_{X_\mathbb{C}}}(\mathcal{E})
\end{equation}
for some locally free sheaf $\mathcal{E}$ on $X_\mathbb{C}$. This fact is due to Tsen's theorem. The sheaf $\mathcal{E}$ induces a vector bundle $E$ on $S$, especially we have $A\cong End(E)\cong E^{\vee}\otimes E$ on $S$. 

Choosing a Hermitian metric $h$ on $E$ induces the metric $h^{\vee}\otimes h$ on $E^{\vee}\otimes E$, and hence also a metric $h^{\mathcal{A}}$ on $A$. This leads to the following definition:

\begin{defi}
Let $X$ be an arithmetic surface. A Hermitian Azumaya algebra $\overline{\mathcal{A}}$ over $X$ is a pair $(\mathcal{A},h^{\mathcal{A}})$, here $\mathcal{A}$ is Azumaya algebra on $X$ and $h^{\mathcal{A}}$ a Hermitian metric on the associated vector bundle constructed as described above.
\end{defi}

\begin{thm}
Let $X$ be an arithmetic surface and let $\overline{\mathcal{A}}$ be a Hermitian Azumaya algebra on $X$, then we have:
\begin{equation*}
2\,\widehat{c}_1(\overline{\mathcal{A}})=0 \in \widehat{CH^{1}}(X).
\end{equation*}
\end{thm}

\begin{proof}
The trace pairing for the algebra $\mathcal{A}$ on $X$, given by
\begin{equation*}
\text{tr}: \mathcal{A} \rightarrow \mathcal{A}^{\vee},\,\,\, a\mapsto (b\mapsto \text{tr}(ba)),
\end{equation*}
is an isomorphism. This is can be checked \'etale locally, where it reduces to the fact that $\mathcal{A}$ becomes a matrix algebra and this algebra is self-dual with respect to the trace map.

Now the Hermitian Azumaya algebra $\overline{\mathcal{A}}$ induces a Hermitian vector bundle $\overline{\mathcal{A}^{\vee}}$ via the dual metric. Theorem \ref{trace} says that the trace induces an isometric isomorphism 
\begin{equation*}
\text{tr}: A\cong End(E)\xrightarrow{\sim} End(E)^{\vee}\cong A^{\vee}.
\end{equation*}
by our choice of the metric on the Hermitian Azumaya algebra.

We especially have
\begin{equation*}
\text{tr}^{*}((h^{\mathcal{A}})^{\vee})=h^{\mathcal{A}},
\end{equation*}
i.e. $h^{\mathcal{A}}$ is the metric induced by the trace map and $(h^{\mathcal{A}})^{\vee}$.

On the one hand we now compute
\begin{equation*}
\widehat{c}_1(\overline{\mathcal{A}^{\vee}})=\widehat{c}_1(\overline{\mathcal{A}}).
\end{equation*}
This follows from the above and \cite[Proposition 1.2.5, Theorem 4.8.(ii)]{soul} since $h^{\mathcal{A}}$ is the metric induced by the metric on $\mathcal{A}^{\vee}$ via the trace map. 

On the other hand we have, using \cite[4.9.]{soul},
\begin{equation*}
\widehat{c}_1(\overline{\mathcal{A}^{\vee}})=-\widehat{c}_1(\overline{\mathcal{A}}).
\end{equation*}

Combining these results gives: 
\begin{equation*}
\widehat{c}_1(\overline{\mathcal{A}})=\widehat{c}_1(\overline{\mathcal{A}^{\vee}})=-\widehat{c}_1(\overline{\mathcal{A}}),
\end{equation*}
or equivalently 
\begin{equation*}
2\,\widehat{c}_1(\overline{\mathcal{A}})=0.
\end{equation*}
\end{proof}
\begin{cor}
Let $X$ be an arithmetic surface and let $\overline{\mathcal{A}}$ be a Hermitian Azumaya algebra on $X$, then $\widehat{c}_1(\overline{\mathcal{A}})=0$ if $\widehat{CH^1}(X)$ is torsion-free.
\end{cor}
\begin{cor}
Let $X$ be an arithmetic surface and let $\overline{\mathcal{A}}$ be a Hermitian Azumaya algebra on $X$, then $\widehat{c}_1(\overline{\mathcal{A}})=0\in \widehat{CH^1}(X)_\mathbb{Q}$, especially $\widehat{ch}(\overline{\mathcal{A}})=rk(\mathcal{A})-\widehat{c}_2(\overline{\mathcal{A}})\in \widehat{CH}(X)_\mathbb{Q}$.
\end{cor}

\section{The first arithmetic Chern class of the Deligne pairing}\label{sec3}
Let $X$ be an arithmetic surface. If $\overline{\mathcal{A}}$ is a Hermitian Azumaya algebra on $X$ and $\mathcal{M}$ is a locally projective left $\mathcal{A}$-module, then using \ref{morita} and Morita equivalence, we see that
\begin{equation*}
\mathcal{M}_{\mathbb{C}}=\mathcal{E}\otimes_{\mathcal{O}_{X_{\mathbb{C}}}} \mathcal{M}'
\end{equation*}
for some locally free sheaf $\mathcal{M}'$ on $X_{\mathbb{C}}$ which induces a vector bundle $M'$ on $S$, so the induced vector bundle $M$ on $S$ is given by $M=E\otimes M'$.
 
The vector bundle $E$ still comes with the Hermitian metric $h$ and we furthermore pick a Hermitian metric $h'$ on $M'$. The tensor product metric of $h$ and $h'$ yields a Hermitian metric $h^{\mathcal{M}}:=h\otimes h'$ on $M$. This defines a Hermitian locally free sheaf $\overline{\mathcal{M}}=(\mathcal{M},h^{\mathcal{M}})$, which also has the structure of a locally projective left $\mathcal{A}$-module. This suggests the following definition:
\begin{defi}\label{hermmod}
Let $X$ be an arithmetic surface and let $\overline{\mathcal{A}}$ be a Hermitian Azumaya algebra over $X$. A Hermitian Azumaya module $\overline{\mathcal{M}}$ is a couple $(\mathcal{M},h^{\mathcal{M}})$ where $\mathcal{M}$ is a locally projective $\mathcal{A}$-module and $h^{\mathcal{M}}$ is a Hermitian metric on the associated vector bundle on $S$ chosen as described above.
\end{defi}
Assume $\overline{\mathcal{M}}$ and $\overline{\mathcal{N}}$ are two Hermitian Azumaya modules, then $\mathcal{H}om_{\mathcal{A}}(\mathcal{M},\mathcal{N})$ is locally free as an $\mathcal{O}_X$-module as both modules are locally projective over $\mathcal{A}$. We remember that we have $\mathcal{A_{\mathbb{C}}}\cong \mathcal{E}nd_{\mathcal{O}_{X_{\mathbb{C}}}}(\mathcal{E})$ for some locally free sheaf $\mathcal{E}$ on $X_{\mathbb{C}}$ and $\mathcal{M}_{\mathbb{C}}\cong \mathcal{E}\otimes \mathcal{M}'$ as well as $\mathcal{N}_{\mathbb{C}}\cong \mathcal{E}\otimes \mathcal{N}'$ by Morita equivalence. This equivalence also gives a natural isomorphism 
\begin{equation*}
\mathcal{H}om_{\mathcal{A_{\mathbb{C}}}}(\mathcal{M}_{\mathbb{C}},\mathcal{N}_{\mathbb{C}})\cong \mathcal{H}om_{\mathcal{O}_{X_\mathbb{C}}}(\mathcal{M'},\mathcal{N'}).
\end{equation*}
The vector bundle associated to the last sheaf is isomorphic to $(M')^{\vee}\otimes N'$.

This bundle comes naturally equipped with the Hermitian metric $(h')^{\vee}\otimes h''$, the one given by the Hermitian metrics $h'$ and $h''$ on the vector bundles $M'$ and $N'$. So we also have a Hermitian metric $h^{(\mathcal{M},\mathcal{N})}$ on the vector bundle associated to $\mathcal{H}om_{\mathcal{A}}(\mathcal{M},\mathcal{N})$. This construction defines the Hermitian vector bundle
\begin{equation*}
\overline{\mathcal{H}om_{\mathcal{A}}(\mathcal{M},\mathcal{N})}:=(\mathcal{H}om_{\mathcal{A}}(\mathcal{M},\mathcal{N}),h^{(\mathcal{M},\mathcal{N})}).
\end{equation*}
The trace map and the tensor-hom adjunction give the following natural isomorphism of locally free $\mathcal{O}_X$-modules:
\begin{align*}
\mathcal{A}\otimes\mathcal{H}om_{\mathcal{A}}(\mathcal{M},\mathcal{N})&\cong \mathcal{A}^{\vee}\otimes\mathcal{H}om_{\mathcal{A}}(\mathcal{M},\mathcal{N})\\
&\cong\mathcal{H}om_{\mathcal{A}}(\mathcal{A}\otimes\mathcal{M},\mathcal{N})\\
&\cong\mathcal{H}om_{\mathcal{O}_X}(\mathcal{M},\mathcal{H}om_\mathcal{A}(\mathcal{A},\mathcal{N}))\cong \mathcal{H}om_{\mathcal{O}_X}(\mathcal{M},\mathcal{N}).
\end{align*}

\begin{lem}\label{natiso}
The natural isomorphism $\mathcal{A}\otimes\mathcal{H}om_{\mathcal{A}}(\mathcal{M},\mathcal{N})\rightarrow \mathcal{H}om_{\mathcal{O}_X}(\mathcal{M},\mathcal{N})$ induces an isometry of Hermitian vector bundles
\begin{equation*}
\overline{\mathcal{A}\otimes\mathcal{H}om_{\mathcal{A}}(\mathcal{M},\mathcal{N})}\xrightarrow{\sim} \overline{\mathcal{H}om_{\mathcal{O}_X}(\mathcal{M},\mathcal{N})}.
\end{equation*}
\end{lem}
\begin{proof}
Looking at all the isomorphisms and the naturally induced metrics on the vector bundles on $S$, this boils down to the following isometry
\begin{equation*}
(E^{\vee}\otimes E\otimes M'^{\vee}\otimes N',h^{\vee}\otimes h\otimes(h')^\vee\otimes h'')\cong ((E\otimes M')^{\vee}\otimes E\otimes N',(h\otimes h')^{\vee}\otimes h\otimes h'')
\end{equation*}
which follows from \ref{dual} and \ref{twist}.
\end{proof}

\begin{cor}\label{cha}
Let $X$ be an arithmetic surface and let $\overline{\mathcal{A}}$ be a Hermitian Azumaya algebra on $X$, if $\overline{\mathcal{M}}$ and $\overline{\mathcal{N}}$ are Hermitian Azumaya modules, the we have the following equality:
\begin{equation*}
\widehat{ch}(\overline{\mathcal{A}})\,\widehat{ch}(\overline{\mathcal{H}om_{\mathcal{A}}(\mathcal{M},\mathcal{N})})=\widehat{ch}(\overline{\mathcal{H}om_{\mathcal{O}_X}(\mathcal{M},\mathcal{N})})
\end{equation*}
\end{cor}
\begin{proof}
By \ref{natiso} the Hermitian vector bundles $\overline{\mathcal{A}\otimes\mathcal{H}om_{\mathcal{A}}(\mathcal{M},\mathcal{N})}$ and $\overline{\mathcal{H}om_{\mathcal{O}_X}(\mathcal{M},\mathcal{N})}$ are isometric, so the properties of the arithmetic Chern classes give the desired result. 
\end{proof}

\begin{defi}
Let $X$ be an arithmetic surface and let $\overline{\mathcal{A}}$ be a Hermitian Azumaya algebra on $X$, then we define the arithmetic $\mathcal{A}$-Chern character of a Hermitian Azumaya module $\overline{\mathcal{M}}$ by:
\begin{equation*}
\widehat{ch}^{\mathcal{A}}(\overline{\mathcal{M}}):=\widehat{ch}(\overline{\mathcal{M}})\,\widehat{ch}(\overline{\mathcal{A}})^{-\frac{1}{2}}
\end{equation*}
and the first arithmetic $\mathcal{A}$-Chern class by
\begin{equation*}
\widehat{c}_1^{\mathcal{A}}(\overline{\mathcal{M}}):=(\widehat{ch}^{\mathcal{A}}(\overline{\mathcal{M}}))^{(1)}.
\end{equation*}
\end{defi}

\begin{rem}\label{form}
By the definition of the arithmetic $\mathcal{A}$-Chern character and by the choice of the induced metrics on all sheaves involved, we see:
\begin{equation*}
\widehat{ch}(\overline{\mathcal{H}om_{\mathcal{A}}(\mathcal{M},\mathcal{N})})=\widehat{ch}^{\mathcal{A}}(\overline{\mathcal{M}})\,\widehat{ch}^{\mathcal{A}}(\overline{\mathcal{N}}).
\end{equation*}
Explicitly we have
\begin{equation*}
\widehat{ch}(\overline{\mathcal{A}})^{-\frac{1}{2}}=\frac{1}{\sqrt{rk(\mathcal{A})}}-\frac{1}{2\sqrt{rk(\mathcal{A})^3}}\widehat{c}_2(\overline{\mathcal{A}}).
\end{equation*}
This computation shows: 
\begin{equation*}
\widehat{c}_1^{\mathcal{A}}(\overline{\mathcal{M}})=\frac{1}{\sqrt{rk(\mathcal{A})}}\widehat{c}_1(\overline{\mathcal{M}}).
\end{equation*}
A similar explicit computation is possible for $\widehat{ch}^{\mathcal{A}}(\overline{\mathcal{M}})$. 

Since $\mathcal{A}$ is an Azumaya algebra, $rk(\mathcal{A})$ is a square, so all these classes are well defined in $\widehat{CH}(X)_\mathbb{Q}$.
\end{rem}

Now we want to compute some arithmetic Chern classes, using the arithmetic Riemann-Roch theorem, due to Gillet and Soul\'e, see \cite{soul2}. As $X$ is an arithmetic surface, we have  $dim(X)-dim(Y)=1$, so we can use the following version of the arithmetic Riemann Roch theorem for a Hermitian vector bundle $\overline{\mathcal{E}}$ on $X$, see \cite[Conjecture 1.5]{soul5}:

\begin{equation}\label{arr}
\widehat{c}_1(\overline{\lambda(\mathcal{E})})=\pi_{*}(\widehat{ch}(\overline{\mathcal{E}})\widehat{Td}(\overline{\mathcal{T}_{X/Y}}))^{(1)}-a(rk(\mathcal{E})(1-g)(4\zeta'(-1)-\frac{1}{6}))
\end{equation}
Here $\overline{\lambda(\mathcal{E})}$ is the 'determinant of the cohomology', a Hermitian line bundle on $Y$ defined by 
\begin{center}
$\lambda(\mathcal{E}):=\bigotimes\limits_{i\geq 0}det(R^i\pi_{*}\mathcal{E})^{(-1)^i}$, 
\end{center}
and the line bundle $\lambda(\mathcal{E})$ is equipped with the Quillen metric $h_Q$ induced by the Hermitian metric on $\mathcal{E}$. Furthermore $g$ is the genus of $X_\mathbb{Q}$ and $\zeta'(-1)$ is the value of the derivative of the Riemann zeta function at $-1$. 

\begin{defi}
Let $X$ be an arithmetic surface and $\overline{\mathcal{A}}$ be a Hermitian Azumaya algebra on $X$. If $(\overline{\mathcal{M}},\overline{\mathcal{N}})$ is a pair of Hermitian Azumaya modules, then we define the $\mathcal{A}$-Deligne pairing of the pair $\overline{\left\langle \mathcal{M},\mathcal{N}\right\rangle_{\mathcal{A}}}$ as the Hermitian line bundle on $Y$ given by:
\begin{equation*}
\overline{\left\langle \mathcal{M},\mathcal{N}\right\rangle_{\mathcal{A}}}:=\overline{\lambda_{\mathcal{A}}(\mathcal{M},\mathcal{N})}\otimes \overline{\lambda_{\mathcal{A}}(\mathcal{M},\mathcal{A})}^{(-1)}\otimes \overline{\lambda_{\mathcal{A}}(\mathcal{A},\mathcal{N})}^{(-1)}\otimes \overline{\lambda_{\mathcal{A}}(\mathcal{A},\mathcal{A})}\,,
\end{equation*}
where
\begin{equation*}
\overline{\lambda_{\mathcal{A}}(\mathcal{M},\mathcal{N})}=\overline{\lambda(\mathcal{H}om_{\mathcal{A}}(\mathcal{M},\mathcal{N}))}.
\end{equation*}
\end{defi}

\begin{thm}
Let $X$ be an arithmetic surface and let $\overline{\mathcal{A}}$ be a Hermitian Azumaya algebra on $X$. If $\overline{\mathcal{M}}$ and $\overline{\mathcal{N}}$ are Hermitian $\mathcal{A}$-line bundles, that is $rk(\mathcal{A})=rk(\mathcal{M})=rk(\mathcal{N})$, then there is the following equality in $\widehat{CH^1}(X)_\mathbb{Q}$: 
\begin{equation*}
\widehat{c}_1(\overline{\left\langle \mathcal{M},\mathcal{N} \right\rangle_\mathcal{A} })=-\pi_{*}(\widehat{c}_1^{\mathcal{A}}(\overline{\mathcal{M}})\,\widehat{c}_1^{\mathcal{A}}(\overline{\mathcal{N}})).
\end{equation*}
\end{thm}

\begin{proof}
By the properties of $\widehat{c}_1$ we have:
\begin{equation*}
\widehat{c}_1(\overline{\left\langle \mathcal{M},\mathcal{N} \right\rangle_\mathcal{A} })=\widehat{c}_1(\overline{\lambda_{\mathcal{A}}(\mathcal{M},\mathcal{N})})- \widehat{c}_1(\overline{\lambda_{\mathcal{A}}(\mathcal{M},\mathcal{A})})- \widehat{c}_1(\overline{\lambda_{\mathcal{A}}(\mathcal{A},\mathcal{N})})+\widehat{c}_1(\overline{\lambda_{\mathcal{A}}(\mathcal{A},\mathcal{A})}).
\end{equation*}

Using the arithmetic Riemann-Roch theorem \ref{arr}, \ref{cha} and \ref{form} one gets:
\begin{equation*}
\widehat{c}_1(\overline{\left\langle \mathcal{M},\mathcal{N} \right\rangle_\mathcal{A} })=\pi_{*}((\widehat{ch}(\overline{\mathcal{M}^{\vee}})-\widehat{ch}(\overline{\mathcal{A}^{\vee}}))(\widehat{ch}(\overline{\mathcal{N}})-\widehat{ch}(\overline{\mathcal{A}}))\widehat{ch}(\overline{\mathcal{A}})^{-1}\widehat{Td}(\overline{\mathcal{T}_{X/Y}}))^{(1)}
\end{equation*}
since the analytic terms of the form $a(-)$ cancel each other.

Using the definitions of $\widehat{ch}$ one computes:
\begin{equation*}
\widehat{c}_1(\overline{\left\langle \mathcal{M},\mathcal{N} \right\rangle_\mathcal{A}})= \pi_{*}(\widehat{c}_1(\overline{\mathcal{M}^{\vee}})\widehat{c}_1(\overline{\mathcal{N}})\widehat{ch}(\overline{\mathcal{A}})^{-1}\widehat{Td}(\overline{\mathcal{T}_{X/Y}}))^{(1)}
\end{equation*}
since $rk(\mathcal{A})=rk(\mathcal{M})=rk(\mathcal{N})$.

So using the definition of $\widehat{c}_1^{\mathcal{A}}$ and $\widehat{Td}$ we finally get:
\begin{equation*}
\widehat{c}_1(\overline{\left\langle \mathcal{M},\mathcal{N} \right\rangle_\mathcal{A}})= -\pi_{*}(\widehat{c}_1^{\mathcal{A}}(\overline{\mathcal{M}})\widehat{c}_1^{\mathcal{A}}(\overline{\mathcal{N}})).
\end{equation*}
\end{proof}

\begin{rem}
The sign in the formula is not surprising, since the Deligne pairing for Hermitian Azumaya modules is the dual of the classical Deligne pairing in the case $\overline{\mathcal{A}}=\overline{\mathcal{O}_X}$, see \cite[Theorem 2.7.]{reede}.
\end{rem}

\bibliography{Artikel}
\bibliographystyle{alphaurl}

\end{document}